\documentclass[a4paper,11pt]{amsart}

\usepackage{latexsym}
\usepackage{amssymb,amsmath}
\usepackage{graphics}
\usepackage{url}
\usepackage[vcentermath]{youngtab}
  
\usepackage{booktabs}  

\newtheorem{theorem}{Theorem}[section]
\newtheorem*{theorem*}{Theorem}
\newtheorem{corollary}[theorem]{Corollary}

\newtheorem{proposition}[theorem]{Proposition}
\newtheorem{conjecture}[theorem]{Conjecture}

\newtheorem{lemma}[theorem]{Lemma}

\theoremstyle{definition}

\newtheorem{example}[theorem]{Example}
\newtheorem{problem}[theorem]{Problem}

\newcommand{\N}{\mathbf{N}}
\newcommand{\Z}{\mathbf{Z}}
\newcommand{\Q}{\mathbf{Q}}
\newcommand{\R}{\mathbf{R}}
\newcommand{\C}{\mathbf{C}}
\newcommand{\F}{\mathbf{F}}

\renewcommand{\P}{\mathcal{P}}
\renewcommand{\R}{\mathcal{O}}

\renewcommand{\O}{\mathcal{O}}
\newcommand{\e}{\mathrm{e}}

\renewcommand{\epsilon}{\varepsilon}

\DeclareMathOperator{\hcf}{hcf}

\DeclareMathOperator{\GL}{GL}

\DeclareMathOperator{\PGL}{PGL}

\DeclareMathOperator{\Out}{Out}
\DeclareMathOperator{\PGammaL}{P\Gamma L}

\newcommand{\Ind}{\big\uparrow}
\newcommand{\Res}{\big\downarrow}
\newcommand{\ind}{\!\uparrow}

\renewcommand{\theta}{\vartheta}

\newcommand{\Gal}{\mathrm{Gal}}

\newcommand{\mfrac}[2]{{\textstyle\frac{#1}{#2}}}

\newcounter{thmlistcnt}
\newenvironment{thmlist}%
	{\setcounter{thmlistcnt}{0}%
	\begin{list}{\emph{(\roman{thmlistcnt})}}{%
		\usecounter{thmlistcnt}%
		\setlength{\topsep}{0pt}%
		\setlength{\leftmargin}{0pt}%
		\setlength{\itemsep}{0pt}%
		\setlength{\labelwidth}{17pt}
		\setlength{\itemindent}{30pt}}%
	}%
	{\end{list}}%

\newcounter{thmlistcnta}
	{\setcounter{thmlistcnta}{0}%
	\begin{list}{\emph{(\alph{thmlistcnta})}}{%
		\usecounter{thmlistcnta}%
		\setlength{\topsep}{0pt}%
		\setlength{\leftmargin}{0pt}%
		\setlength{\itemsep}{0pt}%
		\setlength{\labelwidth}{17pt}
		\setlength{\itemindent}{30pt}}%
	}%
	{\end{list}}%

\newcounter{thmlistcnte}
	{\setcounter{thmlistcnte}{0}%
	\begin{list}{\emph{(\roman{thmlistcnte})}}{%
		\usecounter{thmlistcnte}%
		\setlength{\topsep}{0pt}%
		\setlength{\leftmargin}{50pt}%
		\setlength{\itemsep}{0pt}%
		\setlength{\labelwidth}{30pt}
		\setlength{\itemindent}{0pt}}%
	}%
	{\end{list}}%

\newcommand{\m}{m}	
\newcommand{\dB}{d}
\newcommand{\s}{\sigma}
\renewcommand{\a}{a}
\renewcommand{\e}{e}
\newcommand{\X}{E}

\newcommand{\hstar}{{\hskip0.5pt\star}}

\newcommand{\mbf}[1]{#1}
\newcommand{\mbfr}[1]{\mathbf{#1}}
\newcommand{\simmod}{\raisebox{6pt}{\rotatebox{180}{$\simeq$}}}

\makeatletter
\renewcommand*\env@matrix[1][*\c@MaxMatrixCols c]{%
  \hskip -\arraycolsep
  \let\@ifnextchar\new@ifnextchar
  \array{#1}}
\makeatother

\subjclass[2010]{20B05, secondary 20B15, 20C20}

\begin{document}
\title[Permutation groups containing a regular abelian subgroup]{Permutation groups containing a regular abelian subgroup: the tangled history
of two mistakes of Burnside}
\date{\today}
\author{Mark Wildon}

\maketitle
\thispagestyle{empty}

\begin{abstract}
A group $K$ is said to be a B-group if every permutation group containing $K$ as a regular
subgroup is either imprimitive or $2$-transitive. In the second edition of his
influential textbook on finite groups, Burnside published a proof that cyclic groups
of composite prime-power degree are B-groups. Ten years  later in 1921 he published a proof
that every abelian group of composite degree is a B-group. Both proofs 
are character-theoretic and both have serious flaws. Indeed, the second result is false.
In this note we explain these flaws and prove that every cyclic group of composite
order is a B-group, using only Burnside's character-theoretic methods. We also
survey the related literature, prove some new results on B-groups of prime-power order,
state two related open problems and present some new computational data.
\end{abstract}

\section{Introduction}
In 1911, writing in \S 252 of the second edition of his influential textbook~\cite{BurnsideBook1911}, 
Burnside claimed a proof of the following theorem.

\begin{theorem}\label{thm:main}
Let $G$ be a transitive permutation group of composite prime-power degree
containing a regular cyclic subgroup.
Either $G$ is imprimitive or $G$ is $2$-transitive.
\end{theorem}

An error in the penultimate sentence of Burnside's proof was noted in
\hbox{\cite[page~24]{BurnsideCWorksI}}, where 
Neumann remarks
`Nevertheless, the theorem is certainly true and can be proved by similar character-theoretic
methods to those that Burnside employed'. 
In \S\ref{sec:preliminaries} we present the correct part
of Burnside's proof
in today's language. In \S\ref{sec:main} we prove Theorem~\ref{thm:main} 
using the lemma on cyclotomic
integers in~\S\ref{sec:lemma} below. 
In \S\ref{sec:mainAlt} we build on the correct part of Burnside's proof in a different way,
obtaining an entirely
character-theoretic proof of the following variation on Theorem~\ref{thm:main}.

\begin{theorem}\label{thm:mainAlt}
Let $G$ be a transitive permutation group of composite non-prime-power degree
containing a regular cyclic subgroup. Either $G$ is imprimitive or $G$ is $2$-transitive.
\end{theorem}

In honour of Burnside, Wielandt \cite[\S 25]{WielandtBook1964}
defined a \emph{B-group} to be a group~$K$ such
that every permutation group containing $K$ as a regular subgroup
is either imprimitive or $2$-transitive.
Thus Theorems~\ref{thm:main} and~\ref{thm:mainAlt}  imply that
cyclic groups of composite 
order are B-groups.


The early attempts to prove this result by character-theoretic methods are rich
with interest, but also ripe with errors. Our second aim, which occupies~\S\ref{sec:history},
is to untangle this mess. We end in~\S\ref{sec:abelian} with
some new results on abelian B-groups which require the Classification Theorem
of Finite Simple Groups. We state an open problem on when $C_2^n$ is a B-group, present
a partial solution, consider B-groups of prime-power order and make some further
(much more minor) corrections to the literature.

At a late stage in this work, the author learned of \cite{Knapp}, in which Knapp
gives another way to fix Burnside's proof of Theorem~\ref{thm:main}, using
essentially the same lemma as in~\S\ref{sec:lemma}.
The key step in Knapp's proof is his Proposition~3.1. It
uses two compatible actions of the Galois group of $\Q(\zeta) : \Q$, where $\zeta$ is
a root of unity of order  the degree of $G$: firstly on 
the set permuted by~$G$, and secondly
on the corresponding permutation module. The proof of Theorem~\ref{thm:main} given here
uses only the second action (in a simple way that is isolated in the second step),
and is more elementary in several other respects.
The inductive approach in our  third step is also new. Given
the historical importance of Theorem~\ref{thm:main}, the author
believes it is worth putting this shorter proof on record. Theorem~\ref{thm:mainAlt}
is not proved in~\cite{Knapp}.

\section{Lemma on cyclotomic integers}
\label{sec:lemma}

The following lemma is essentially the same as
 Lemma~4.1 in \cite{Knapp}. A proof is included for completeness.
Recall that the degree of 
the extension of~$\Q$ generated by
a primitive $d$-th root of unity is $\phi(d)$, where $\phi$ is Euler's totient function.

\begin{lemma}\label{lemma:cyclotomic}
 Let $p$ be a prime and let $n \in \N$. 
For each $r$ such that $1 \le r < p^{n-1}$, let
\[ R(r) = \{ r, r + p^{n-1}, \ldots, r + (p-1)p^{n-1} \}. \] 
Let $\zeta$ be a primitive $p^n$-th root of unity and
let $\omega = \zeta^{p^{n-1}}$. 
If $\sum_{i=0}^{p^n-1} \a_i \zeta^i \in \Q[\omega]$
where $\a_i \in \Q$ for each $i$,
then the coefficients $\a_i$ are constant for $i$ in each set $R(r)$.
\end{lemma}

\begin{proof}
By the Tower Law 
$[\Q(\zeta) : \Q(\omega)] = [\Q(\zeta) : \Q]/[\Q(\omega) : \Q] = \phi(p^{n}) / \phi(p) = (p-1)p^{n-1}/(p-1)
= p^{n-1}$. Therefore $\Psi(X) = X^{p^{n-1}}-\omega$ 
is the minimal polynomial of~$\zeta$ over $\Q(\omega)$.
By hypothesis there exists $\gamma \in \Q[\omega]$ such that 
\[ f(X) = -\gamma + \sum_{0 \le i < p^n} \a_i X^i \]
has $\zeta$ as a root. Hence $f(X)$ is divisible in $\Q(\omega)[X]$ by $\Psi(X)$. 
There is a unique expression $f(X) = f_0(X) + \sum_{0 < r < p^{n-1}} f_r(X)$
where
\[ f_r(X) = \sum_{0 \le i < p^n \atop i \equiv r \;\mathrm{mod}\; p^{n-1}} \a_i X^i \]
for $0 < r < p^{n-1}$.
The remainder when~$X^d$ is divided by $\Psi(X)$ has non-zero coefficients
only for those~$X^c$ such that $c$ is congruent to $d$ modulo $p^{n-1}$. 
Therefore each $f_r(X)$ is divisible by $\Psi(X)$ and so
$f_r(\zeta) = 0$ for each $r$. Since the coefficients of $f_r$ for $0 < r <p^{n-1}$ are rational,
it follows that each such $f_r$ is divisible, \emph{now in $\Q[X]$}, by the minimal
polynomial of~$\zeta$ over $\Q$, namely 
$\Phi_{p^n}(X) = 1 + X^{p^{n-1}} 
+ \cdots + X^{(p-1)p^{n-1}}$. 
Since $f_r$ has degree at most $p^n-1$, this implies that
$f_r(X) = b_rX^r \Phi_{p^n}(X)$ for some $b_r \in \Q$. The lemma follows.
\end{proof}

\section{Burnside's method: preliminary results} 
\label{sec:preliminaries}

We may suppose that $G$ acts on $\{0,1,\ldots, d-1\}$, where $d \in \N$ 
is composite, and that $g = (0,1,\ldots,d-1)$ is a $d$-cycle in $G$. Let $H$ be the point stabiliser of $0$.
Let $M = \langle e_0, e_1 \ldots, e_{d-1} \rangle_\C$
be the natural permutation module for $G$. Let $\zeta$ be a primitive $d$-th root of unity and for $0 \le j < d$ let
\begin{equation}\label{eq:v} v_j = \sum_{0 \le i < d} \zeta^{-ij} e_i. \end{equation}
Since $e_i g = e_{i+1}$, where subscripts are taken modulo $d$, we have
$v_j g = \zeta^j v_j$ for each $j$.  Note that $v_0 = \sum_{0 \le i < d} e_i$ spans the (unique) trivial
$\C G$-module of $M$. Let
\begin{equation}\label{eq:unique} M = \langle v_0 \rangle \oplus V_1 \oplus \cdots \oplus V_t 
\end{equation}
be a direct sum decomposition of $M$ into irreducible $\C G$-submodules.
The~$v_j$ are eigenvectors of $g$ with distinct eigenvalues. Therefore they form
a basis of $M$. Moreover, since the eigenvalues are distinct,
each of the summands $V_1, \ldots, V_t$ 
has a basis consisting of some of the $v_j$. Thus the decomposition in~\eqref{eq:unique}
is unique. For each summand $V_k$,
let $B_k = \{ j : 0 < j < p^n, v_j \in V_k\}$. 
Let $\phi_k$ be the character of $V_k$.

The following two lemmas are the key observations in Burnside's method.

\begin{lemma}\label{lemma:invar}
For each $k$ such that $1 \le k \le t$, the vector $\sum_{j \in B_k} v_j$ is $H$-invariant.
\end{lemma}

\begin{proof}
The permutation character $\pi$ of~$G$ is $1_G + \sum_{k=1}^t \phi_k$, where the summands
are distinct and irreducible.
By Frobenius reciprocity we have
\[ 1 = \langle \pi, \phi_k \rangle_G = \langle 1_H \Ind^G, \phi_k \rangle_G = 
\langle 1_H, \phi_k\Res_H \rangle_H \]
for each $k$.
Therefore each $V_k$ has a unique $1$-dimensional $\C H$-invariant submodule. 
Since $e_0 = \mfrac{1}{p^d}\sum_{0 \le j < d} v_j$ is $H$-invariant, and the projection of~$e_0$ into $V_k$
is $\mfrac{1}{p^d}\sum_{j \in B_k} v_j$, this submodule is spanned
by $\sum_{j \in B_k} v_j$.
\end{proof}

\begin{lemma}\label{lemma:firstStep}
If $\mathcal{O}$ is an orbit of $H$ on $\{0,1,\ldots,d-1\}$ 
and $1 \le k \le t$ then the sum
$\sum_{i \in \mathcal{O}} \zeta^{ij}$
is constant for $j \in B_k$.
\end{lemma}

\begin{proof}
Observe that $\sum_{k \in \mathcal{O}} e_k$ is $H$-invariant.
An easy calculation (which may be
replaced by the observation that the character table of $C_{d}$ is an orthogonal matrix)
shows that $e_i = \mfrac{1}{p^d}\sum_{0 \le j < d} \zeta^{ij} v_j$ for each $i$. Therefore
\[ \sum_{i \in \mathcal{O}} e_i =  \sum_{0 \le j < d}
\Bigl( \sum_{i \in \mathcal{O}}\zeta^{ij} \Bigr) v_j. \]
By Lemma~\ref{lemma:invar} the coefficients are constant for $j \in B_k$.
\end{proof}

The following proposition is used in the final step of the proof of both main theorems.

\begin{proposition}\label{prop:imprimitive}
If there is a prime $p$ dividing $d$ and a 
summand $V_k$ 
whose basis $\{v_j : j \in B_k\}$ contains only basis
vectors $v_j$ with $j$ divisible by $p$ then
there exists a normal subgroup of $G$ containing $g^{d/p}$ 
whose orbits
form a non-trivial block system. 
\end{proposition}

\begin{proof}
Let $N$ be the
kernel of $G$ acting on $V_k$. Since $v_j g = \zeta^j v_j$,~$N$ contains~$g^{d/p}$. By 
Lemma~\ref{lemma:invar},~$V_k$ has 
$\langle \sum_{j \in B_k} \!v_j \rangle$ as an $HN$-invariant subspace. Since $V_k$ is 
not the trivial module, 
we have $HN < G$. Hence $N$ is non-trivial but intransitive.
The orbits of the normal subgroup $N$ are blocks 
of imprimitivity for~$G$.
\end{proof}


\section{Proof of Theorem~\ref{thm:main}}
\label{sec:main}

\subsection*{First step}
By hypothesis $G$ has degree $p^n$ where $p$ is prime and \hbox{$n \ge 2$}.
The Galois group $\Gal(\Q(\zeta) : \Q)$ of the field extension \hbox{$\Q(\zeta) : \Q$} 
permutes the basis vectors $v_j$ while preserving the unique direct sum decomposition~\eqref{eq:unique}.
Hence $\Gal(\Q(\zeta) : \Q)$ permutes
the sets $B_1, \ldots, B_t$. 
By Proposition~\ref{prop:imprimitive},
we may assume that every $B_k$ contains some $j$ not divisible by $p$.
Hence, given any $\m$ 
such that $0 <\m < n$,
there exists $j$ 
not divisible by $p$ such that the set $B_k$ containing $p^m$ also contains $j$.
Let $B_\ell$ be the set containing $1$.
Since the Galois group is transitive on
$\{\zeta^j : 0 < j < p^n, p \nmid j\}$, by conjugating $\zeta^j$ to $\zeta$,
we see that $p^\m c \in B_\ell$ 
for some $c$ not divisible by $p$.

Let $\mathcal{P}$ be the partition of 
$\{1,\ldots,p^n-1\}$ into the orbits of $H$ other than~$\{0\}$.
The previous paragraph and Lemma~\ref{lemma:firstStep} imply that
for all $m$ such that $0 < \m < n$ there
exists $c_\m \in \N$, not divisible by $p$, such that
\begin{equation}\label{eq:key} 
\sum_{i \in \mathcal{O}} \zeta^{i} = \sum_{i \in \mathcal{O}} \zeta^{p^\m c_\m i} \end{equation}
for each $\mathcal{O} \in \mathcal{P}$.

\subsection*{Second step} 
We shall show by induction on $n$ that~\eqref{eq:key} implies
that $\mathcal{P}$ is the one-part partition. It then follows that
 $H$ is transitive on $\{1,\ldots,p^n-1\}$
and so $G$ is $2$-transitive, as required. 

Fix $\mathcal{O} \in \mathcal{P}$. 
Taking $\m=n-1$ in~\eqref{eq:key} and applying Lemma~\ref{lemma:cyclotomic} with $\omega = \zeta^{p^{n-1} 
c_{n-1}}$, we find that the coefficients in $\sum_{i \in \mathcal{O}} \zeta^i$ are constant 
on the sets $R(r) = \{r, r+p^{n-1}, \ldots, r+(p-1)p^{n-1} \}$ for $0 < r < p^{n-1}$. Hence
$\mathcal{O}$ is a union of some of these sets,
together with some of $\{p^{n-1}\}$, \ldots, $\{(p-1)p^{n-1}\}$. The contributions
from $R(r)$ to~\eqref{eq:key} are 
\begin{align}
\sum_{i \in R(r)} \zeta^{i} &= 0,  \label{eq:zetasum} \\
\sum_{i \in R(r)} \zeta^{p^\m c_\m i} &= p \zeta^{p^\m c_\m r}.  \label{eq:zetasum2}
\end{align}

\subsubsection*{Case $n=2$.} Let $\omega = \zeta^{pc_1}$.
Taking $m=1$ in~\eqref{eq:key} and
substituting the relations in~\eqref{eq:zetasum} and~\eqref{eq:zetasum2} 
we get
\[ \sum_{ r \in \mathcal{O} \atop {0 < r <p}} 0 + 
\sum_{pi \in \mathcal{O}} \omega^i = \sum_{ r \in \mathcal{O} \atop 0 < r < p }
p \omega^r + \sum_{pi \in \mathcal{O} \atop 0 < i < p  } 1. \] 
This rearranges to 
\[ \bigl|\{\mathcal{O} \cap \{p,2p,\ldots,(p-1)p\}\bigr| +
\sum_{0 < i < p} (p[i \in \mathcal{O}] - [pi \in \mathcal{O}]) \omega^i = 0,
\]
where the Iverson bracket
$[P]$ is $1$ if the statement $P$ is true, and $0$ if false.
Since the minimal polynomial of $\omega$, namely $1 + X + \cdots + X^{p-1}$, has degree $p-1$
and constant
coefficients,
it follows that 
$\bigl|\{\mathcal{O} \cap \{p,\ldots,(p-1)p\}\bigr| = p-1$
and $i \in \mathcal{O}$ for each $i$ such that $0 < i < p$. Thus
$\mathcal{O} = \{1,\ldots, p^2-1\}$ as required.

\subsubsection*{Inductive step} 
Let $n \ge 3$. Let $T = \{p^{n-1},\ldots, (p-1)p^{n-1}\}$.
Substituting~\eqref{eq:zetasum2} in the right-hand-side of~\eqref{eq:key} for first $m=1$ and then a general $m$
such that $0 < m < n$, we have
\[ 
\sum_{ r \in \mathcal{O} \atop 0 < r < p^{n-1}} p\zeta^{p c_1r} 
+ |\mathcal{O} \cap T| \, = \!\!\!
\sum_{r \in \mathcal{O} \atop 0 < r < p^{n-1}} p\zeta^{p^\m c_\m r} 
+ |\mathcal{O} \cap T |.  \]
For each 
$\mathcal{O} \in \mathcal{P}$, define 
$\mathcal{O}_\star = \mathcal{O} \cap \{1,\ldots,p^{n-1}-1\}$.
Clearly $\{\mathcal{O}_\star : \mathcal{O} \in \mathcal{P} \}$ is a set partition of $\{1,\ldots, p^{n-1}-1\}$.
Let $\zeta_\star = \zeta^{pc_1}$ and, for each $\m$ such that $0 < \m < n$, 
choose $d_\m \in \N$ such that $c_1 d_\m \equiv c_\m$ mod $p$. We may suppose that $d_1 = 1$.
Replacing $r$ with~$i_\star$, the previous displayed equation implies
\[ \sum_{i_\star \in \mathcal{O}_\star} \zeta_\star^i =
\sum_{i_\star \in \mathcal{O}_\star } \zeta_\star^{p^{\m-1} d_\m i_\star}. \]
Comparing with~\eqref{eq:key}, we
see that all the conditions are met to apply the inductive hypothesis. Hence
$\mathcal{O}_\star = \{1,\ldots,p^{n-1}-1\}$ and so $\mathcal{O}$ contains
$\{1,\ldots,p^n-1\}\backslash T$. By~\eqref{eq:zetasum} and~\eqref{eq:zetasum2} we have
$\sum_{i \in \{1,\ldots, p^n-1\} \backslash T } \zeta^i = 0$ and 

\vspace*{-6pt}
\[ \sum_{0 < i < p^{n-1} \atop i \not\in T}  
\zeta^{pc_1i }
= p \sum_{0 < i < p^{n-1}} \zeta_\star^i = 
 -p. \]
Substituting these two results in the case $\m=1$ of~\eqref{eq:key} we get 
\[ \sum_{p^{n-1}i \in \mathcal{O} \hskip0.5pt\cap\hskip1pt T} \zeta^{p^{n-1}i} = -p + | \mathcal{O} \cap T|. \]
It follows, as in the final step of the case $n=2$, that $|\mathcal{O} \cap T| = p-1$ and so $\mathcal{O} \supseteq T$
and $\mathcal{O} = \{1,\ldots,p^n-1\}$,  as required.

\section{Proof of Theorem~\ref{thm:mainAlt}} 
\label{sec:mainAlt}

We continue from the end of \S\ref{sec:preliminaries}.
Let $\theta : \langle g \rangle \rightarrow \C$ be the faithful linear character of $\langle g\rangle$
defined by $\theta(g) = \zeta$. For $1 \le k \le t$, let
$\pi_k$ be the character of $V_k$ restricted to $\langle g 
\rangle$. Since $\langle v_j \rangle$ affords $\theta^j$, we have
$\pi_k = \sum_{j \in B_k} \theta^j$. Since the sets $B_1, \ldots, B_t$ are disjoint, 
the characters $\pi_k$ are linearly independent.

Let $p$ be a prime dividing $d$.
The character of $V_k^{\otimes p}$ is $\pi_k^p$. Since $(a+b)^p \equiv a^p + b^p$ mod $p$
for all $a, b \in \Z$, we have
\begin{equation}
\label{eq:ppower} \pi_k^p = \sum_{0 \le r < d/p} \bigl| \{j \in B_k : jp \equiv rp \text{ mod $d$} \} \bigr|
\theta_{rp} 
+ p \pi \end{equation}
for some character $\pi$ of $\langle g \rangle$. 
Let $\pi_k^p - p\pi = a 1_H +  \sum_{\ell=1}^t a_\ell \pi_\ell$.
By the linear independence of the $\pi_\ell$, it follows from~\eqref{eq:ppower} that
if $a_\ell \not=0$ then~$\pi_\ell$ contains only characters of the form $\theta_{rp}$ with $1 \le r < d/p$.
Thus for any such $\ell$,~$B_\ell$ contains only basis vectors~$v_j$ with~$j$
divisible by $p$ and, by Proposition~\ref{prop:imprimitive}, $G$ is imprimitive.
We may therefore assume that
$\bigl| \{j \in B_k : jp \equiv rp \text{ mod $d$} \} \bigr|$ is a multiple
of~$p$ for each $r$ such that $1 \le r < d/p$.
Identifying $\{0,1,\ldots, d-1\}$ with $\Z / d\Z$,
note that $jp \equiv rp$ mod~$d$ if and only if 
$j \in r + \langle d/p \rangle$.
Therefore \emph{for each prime~$p$ dividing $d$, 
each $B_k$ is  the union of a subset of $\langle d/p \rangle$ and some proper
cosets $r + \langle d/p \rangle$.} 

Let $q$ be a prime dividing $d$ other than $p$. Since the
subgroups $\langle d/p \rangle$ and $\langle d/q \rangle$ of $\Z/d\Z$ meet in $0$, 
each member of $\langle d/p \rangle \backslash \{0\}$ is in a proper coset
of $\langle d/q \rangle$, and similarly with $p$ and $q$ swapped.
By the conclusion of the previous paragraph, if
$B_k$ meets $\langle d/pq \rangle$ then $B_k$ contains $\langle d/pq \rangle \backslash \{0\}$.
At most one~$B_k$ has this property. 
If $t=1$ then $G$ is $2$-transitive, so we may assume that $d > pq$ and there exists
$B_k$ not meeting $\langle d/pq \rangle$. 
For this $B_k$ there exist $r_1, \ldots, r_s$ such that $0 < r_1 < \ldots < r_s < d/pq$ and
\[ B_k = \bigcup_{\e=1}^s  (r_\e + \langle d/pq \rangle). \]
Thus $|B_k| = s pq$ and
\begin{equation}
\label{eq:square} \overline{\pi_k} \pi_k  = s (\theta_0 + \theta_{d/pq} + \cdots + \theta_{(pq-1)d/pq} ) + \psi\end{equation}
where the coefficient of $\theta_j$ in $\psi$ is equal to the number
of pairs $(\e,\e')$ such that $j \in -r_\e + r_{\e'} + \langle d/pq \rangle$.
There are exactly $s$ such pairs if and only if for all $\e$ there exists a unique
$\e'$ such that $r_\e + j + \langle d/pq\rangle = r_{\e'} + \langle d/pq \rangle$,
or, equivalently, if and only if $B_k + j = B_k$, where the addition is performed in
$\Z/ d\Z$. Let
\[ J = \{ j \in \Z / d\Z : B_k + j = B_k \}. \]
Since $J$ is a subgroup of $\Z / d\Z$ containing $d/pq$
we have $J = \langle m \rangle$ for some~$m$ dividing $d/{pq}$. Since 
 $0 \not\in B_k$, and so $-r_1, \ldots, -r_s \not\in J$, we have $m > 1$.
Thus~\eqref{eq:square} may be rewritten as
\[ \overline{\pi}_k \pi_k = s\bigl( \theta_0 + \theta_{m} + \cdots
+ \theta_{n-m} \bigr) + \phi \] 
where $\langle \phi, \theta_j \rangle < s$ for all $j$ not divisible by $m$.
By the linear independence of $\pi_1, \ldots, \pi_t$, 
there exists $\pi_k$ such that if $\langle \pi_k, \theta_j \rangle > 0$
then $j$ is a multiple of~$m$. The result now follows from Proposition~\ref{prop:imprimitive}.


\section{A historical survey of Burnside's method and B-groups}
\label{sec:history}

\subsection{Burnside's work for prime-power degree}
We begin in 1901 with \cite[\S 7]{Burnside1901}, in which Burnside 
used character-theoretic arguments
to prove the following important dichotomy.
(All of the papers of Burnside discussed below appear in Volume II of his collected works \cite{BurnsideCWorksII}.)

\begin{theorem}[Burnside 1901 \protect{\cite[\S 7]{Burnside1901}}]\label{thm:Burnside1901}
A permutation group of prime degree~$p$ is either $2$-transitive or contains a normal subgroup of order $p$.
\end{theorem}

In the following \S 8 he proves Theorem~\ref{thm:main}
for permutation groups of odd degree $p^2$ using character theory.
He comments `It appears highly probable that this
result may be extended to any group of odd order which contains a regular substitution of order equal to the
degree of the group; but I have not yet succeeded in proving this.' 

In the revised second edition of his textbook~\cite{BurnsideBook1911},
Burnside added five entirely new chapters on linear groups and characters. 
Most notably these include the well-known character-theoretic proof of the $p^aq^b$-Theorem.
In~\S 251 
he used the method of cyclotomic sums and basis sets, introduced in his 1906 paper~\cite[\S 7]{Burnside1906}
but presented in his textbook with some simplifications, to prove Theorem~\ref{thm:Burnside1901}.
The following \S 252, whose correct part was presented in \S\ref{sec:preliminaries},
attempts to prove Theorem~\ref{thm:main}.  Burnside's argument appears
to have been generally accepted, both at the time and later, until Neumann
pointed out the error in his essay in \cite{WielandtWerkeI}.
For example, it is cited
without critical comment by Wielandt in~\cite{WielandtBook1964}.
Its mistake is to assert that the only solutions to~\eqref{eq:key} when $m=n-1$
have $|\mathcal{O}| = p^n-1$. This gives one solution, but
there are others. 

Recall that if $1 \le r < p^{n-1}$ then 
$R(r) = \{ r, r + p^{n-1}, \ldots,
r+ (p-1)p^{n-1} \}$.
Define $Z \subseteq \{1,\ldots,p^n-1\}$ to be \emph{null} if there
exists $s \in \N_0$ and distinct $r_{ij} \in \{1,\ldots, p^{n-1}-1\}$
for $0 \le i \le p-1$ and $1 \le j \le s$
such that $r_{ij} \equiv i$ mod $p$ for each $i$ and $j$ and
$Z = \bigcup_{i=0}^{p-1} \bigcup_{j=1}^s R(r_{ij})$. 


\begin{proposition}
Let $n\ge 2$ and let $\omega$ be a primitive $p$-th root of unity. 
Let $\mathcal{O} \subseteq \{1,\ldots,p^n-1\}$. Then 
\[ \sum_{i \in \mathcal{O}} \zeta^i = \sum_{i \in \mathcal{O}} \omega^i \]
if and only if  \emph{either}
\begin{thmlist}
\item $\mathcal{O}$ is null; \emph{or}
\item $\mathcal{O} = \{p^{n-1}, \ldots, (p-1)p^{n-1} \} \; \cup \;
\bigcup_{i=1}^{p-1} R(r_i) \; \cup \; Z$ where $Z$ is null, 
the $r_i$ are distinct elements of $\{1,\ldots,p^{n-1}-1\}\backslash Z$
 and $r_i \equiv i$ mod $p$ for each $i$.
\end{thmlist}
\end{proposition}

The proof is similar to the inductive step in \S\ref{sec:main}, 
using~\eqref{eq:zetasum} and~\eqref{eq:zetasum2} to show
that if $Z$ is null then $\sum_{i \in Z}\xi^i = \sum_{i \in Z}\omega^i = 0$,
and Lemma~\ref{lemma:cyclotomic} to show that $\mathcal{O} \backslash \{p^{n-1}, \ldots, (p-1)p^{n-1} \}$
is a union of the sets $R(r)$.
Note that since $r_{01} \equiv 0$ mod $p$,
and $1 \le r_{01} < p^{n-1}$, Case (i) is relevant only when $n \ge 3$. The 
smallest possible $\mathcal{O}$ has size $p^2-1$, coming from Case (ii); this shows
Burnside's claim is false whenever $p \ge 3$ or $n \ge 3$. The lack
of structure in the solutions, beyond that captured by the sets $R(r)$, 
suggests that any fix to Burnside's proof must involve significant further ideas.

\subsection{Burnside's 1921 paper}\label{sec:Burnside1921}
In \cite{Burnside1921}, Burnside claimed a `remarkably simple' proof that 
every abelian group 
 that is not elementary abelian is a B-group, as conjectured at the end of 
 \S 252 of \cite{BurnsideBook1911}. (Of course Burnside
did not use the term `B-group'.) The groups $S_\dB \wr S_2$ in their
primitive action for $\dB$ composite, seen in Example~\ref{ex:prim} below, 
show that this result is false.
In \cite[\S 15]{DManning1936}, D.~Manning raised this family 
of counterexamples and observed `the first
and most important part of the proof must contain a serious mistake'.

In today's language, Burnside considers a permutation group $G$ of degree $\dB\dB'$
acting on $\{0,\ldots, \dB-1\} \times \{0,\ldots, \dB'-1\}$,
containing a regular subgroup
$K = \langle g_\dB \rangle \times \langle g'_{\dB'} \rangle$ where $g_\dB = (0,1,\ldots, d-1)$
and \hbox{$g'_{\dB'} = (0,1,\ldots,d'-1)$}. The natural $\C G$-permutation module $M$  factorizes
on restriction to $K$
as $\langle e_0, \ldots, e_{\dB-1} \rangle \otimes \langle e'_0, \ldots, e'_{\dB'-1} \rangle$.
Let $\zeta_\dB$, $\zeta_{\dB'} \in \C$ be primitive roots of unity of orders $\dB$ and~$\dB'$,
respectively.
The analogue of the $v_j$ basis element defined earlier in~\eqref{eq:v} is
\[ v_{(j,j')} = \sum_{0 \le i < \dB} \zeta_d^{-ij}e_i \otimes \sum_{0 \le i' < \dB'}
\zeta_{d'}^{-i'j'}
e'_{i'}
\]
where $0 \le j <\dB$ and $0\le j' < \dB'$. 
As before, $M$ has a unique decomposition $\langle v_{(0,0)} \rangle
\oplus V_1 \oplus \cdots \oplus V_t$ where each irreducible summand $V_k$ has a basis
$\{ v_{(j,j')} : (j,j') \in B_k \}$ for some subset $B_k$ of $\{0,\ldots, \dB-1\} \times 
\{0,\ldots, \dB'-1\}$. Let $\phi_k$ be the character of $V_k$.
The analogue of Lemma~\ref{lemma:firstStep}
is that if $\mathcal{O}$ is an orbit of the point stabiliser
$H$ of $(0,0)$, and $1 \le k \le t$ then 
\begin{equation}\label{eq:BurnsideCO}
\sum_{(i,i') \in \mathcal{O}} \zeta_\dB^{ij} \zeta_{\dB'}^{i'j'} \end{equation} 
is 
constant for $(j,j') \in B_k$. Burnside proves this, and also proves (in a similar way)
the dual relation that the character value 
$\smash{\phi_k(g_\dB^ig_{\dB'}^{\prime i'})=
 \sum_{(j,j') \in B_k} \zeta_\dB^{ij} \zeta_{\dB'}^{i'j'}}$ 
is 
constant for $(i,i') \in \mathcal{O}$. Hence
\begin{align} 
\sum_{(i,i') \in \mathcal{O}} \sum_{(j,j') \in B_k} \zeta_\dB^{ij} \zeta_{\dB'}^{i'j'}
&= |B_k| \sum_{(i,i') \in \mathcal{O}} \zeta_\dB^{ij} \zeta_{\dB'}^{i'j'}\label{eq:Burnside1} \\
&= |\mathcal{O}| \sum_{(j,j') \in B_k} \zeta_\dB^{ij} \zeta_{\dB'}^{i'j'}\label{eq:Burnside2}
\end{align}
provided  $(j,j') \in B_k$ in the right-hand side of~\eqref{eq:Burnside1} and $(i,i') \in \mathcal{O}$
in the right-hand side of~\eqref{eq:Burnside2}.
Burnside chooses $B_k$ to contain $(\dB/q, 0)$
where $q$ is a prime factor of $\dB$ and $\mathcal{O}$ to contain $(1,0)$.
By taking $(j,j') = (\dB/q,0)$ in~\eqref{eq:Burnside1} and $(i,i') = (1,0)$ in~\eqref{eq:Burnside2} he obtains
$|B_k| \sum_{(i,i') \in \mathcal{O}} \zeta_d^{id/q} =
|\mathcal{O}| \sum_{(j,j') \in B_k} \zeta_d^j = |\mathcal{O}|\phi_k(g_d)$, 
and so
\begin{equation}\label{eq:Burnside4}
\phi_k(g_d) = \frac{|B_k|}{|\mathcal{O}|} \sum_{(i,i') \in \mathcal{O}} \omega^{i} \end{equation}
where $\omega = \zeta_d^{d/q}$ is a primitive root of unity of order $q$. 

The
fourth displayed equation on page 484 of \cite{Burnside1921} claims that $\phi_k(g_\dB^q) = |\mathcal{B}_k|$,
and so $g_\dB^q$ is in the kernel of $\phi_k$.
It appears that Burnside substitutes $g_\dB^q$ for~$g_\dB$
in~\eqref{eq:Burnside4}, and replaces $\omega$ with $\omega^q$.
If~\eqref{eq:Burnside4} expressed
~$\phi_k(g_\dB)$ as a sum of 
eigenvalues, as in~\eqref{eq:Burnside2}, this would be legitimate.
However this is not the case, and the following example shows that
Burnside's claim is in general false.

\begin{example}\label{ex:prim}
Let $\dB \in \N$. Let $S$ be the symmetric group on the set $\{0,1,\ldots,\dB-1\}$.
Let $N = S \times S$ and let $G \cong S \wr C_2$ 
be the wreath product $N \rtimes \langle \tau \rangle$
where $\tau$ has order $2$ and acts on $N$ by $(g,g')^\tau = (g',g)$.
In the  action of $G$ on $\{0,1,\ldots, \dB-1\}^2$, the point
stabiliser $H$ of~$(0,0)$, namely $(T \times T) \rtimes \tau$
where $T$ is the symmetric group on $\{1,\ldots, \dB-1\}$, 
has
two non-singleton orbits: $\bigl\{(j,0), (0,j) : 1 \le j < \dB \bigr\}$
and $\bigl\{ (j,j') : 1 \le j, j' < \dB \bigr\}$. Therefore~$G$ is not $2$-transitive.
Provided $\dB \ge 3$, $H$ is a maximal subgroup of $G$, so~$G$ is primitive.
Let $g_\dB = g_\dB' = (0,1,\ldots, \dB-1)$. 
Since $\langle g_\dB \rangle \times \langle g'_\dB \rangle \le N$ 
acts regularly, $C_\dB \times C_\dB$ is not a $B$-group whenever $\dB \ge 3$.

Let $\dB \ge 3$. The natural permutation character of $S$ is $1_S + \chi$ 
where $\chi$ is irreducible. By the branching rule (see \cite[Ch.~9]{James} or
\cite[Lemma~2.3.10]{JK}),
$\chi$ is the unique non-trivial character of~$S$ 
whose restriction to $T$ 
contains the trivial character. By
 the classification of irreducible characters of wreath products
\cite[Theorem 4.3.34]{JK}, it follows that
the irreducible characters of~$G$ that contain the trivial character
on restriction to $H$ are  $1_G$, $\phi$ and 
$\smash{\chi^{\widetilde{\times 2}}}$,
where $\phi = (\chi \times 1_{S}) 
\!\ind_{N}^{\hskip0.5pt G}$ and
 $\smash{\chi^{\widetilde{\times 2}}}$ is the
unique character of $G$ whose restriction
to $N$ is $\chi \times \chi$. By Frobenius reciprocity,
the permutation character of $M$ is 
$1_G + \phi  
+ \smash{\chi^{\widetilde{\times 2}}}$.
Considering restrictions to $\langle g_\dB \rangle \times \langle g'_{\dB} \rangle$, we get
$M = \langle v_{(0,0)} \rangle \oplus \langle v_{(j,0)}, v_{(0,j')} : 1 \le j < \dB, 1 \le j' < \dB
\rangle \oplus \langle v_{(j,j')} : 1 \le j, j' < \dB \rangle$. The second 
summand has  character $\phi$ and contains $v_{(1,0)}$ and $v_{(0,1)}$, so is a faithful $\C G$-module.
Thus, contrary to Burnside's claim, no
non-identity power of~$g_\dB$ is in the kernel of $\phi$.
Burnside's conclusion, that $G$ has a proper
normal subgroup containing~$g_\dB^q$ holds, since we may take the base group $N$,
but clearly Burnside intends the normal subgroup to be the kernel of $\phi$,
so that Proposition~\ref{prop:imprimitive} can be applied,
and
the kernel of $\phi$ is trivial.
\end{example}

The penultimate paragraph of Burnside's paper
considers the case where~$\dB$ and $\dB'$ are distinct primes. 
This is the hardest part of the paper to interpret: the claims are  correct,
but the argument has 
a significant gap.
Burnside has already assumed that $G$ is not
$2$-transitive.
If a basis set 
$B_k$ is contained in $\{(1,0),\ldots, (\dB-1,0)\}$ then, identifying $(j,j')$ with
$\dB'\! j + \dB j'$ mod $\dB\dB'$,
Proposition~\ref{prop:imprimitive} implies that $G$ has a normal intransitive subgroup $N$
containing $\langle g_\dB \rangle$. 
This gives the first of Burnside's claims. While not stated explicitly, it seems
that Burnside then assumes, as he may, 
that no $B_k$ is contained in $\{(1,0),\ldots,(d-1,0)\}$. He makes two further claims,
equivalent to the following:
\begin{itemize}
\item[(i)] If $B_k$ meets $\{(1,0),\ldots, (d-1,0)\}$ then
$B_k$ is a union of sets
of the form $\big\{ (j,0), (j,1), \ldots, (j,\dB'-1) \bigr\}$ where $1 \le j < \dB$. 
\item[(ii)] there is a set $B_\ell$ contained in $\{(0,1), \ldots, (0,\dB'-1)\}$.
\end{itemize}
Clearly (i) implies (ii), and by Proposition~\ref{prop:imprimitive}, (ii) implies that
$G$ has a normal intransitive subgroup $N$
containing $\langle g'_{\dB'} \rangle$. 
To prove (i), we use 
the italicised conclusion of the second paragraph in the proof
of Theorem~\ref{thm:mainAlt} in \S\ref{sec:mainAlt}: taking $p = \dB'$, this implies that $B_k$ is the union
of a subset of $\{(0,1),\ldots,(0,\dB'-1)\}$ and some sets of the required form.
Since
 $[\Q(\zeta_{dd'}) : \Q(\zeta_{d})] = \phi(dd')/\phi(d) =
\phi(d') 
= [\Q(\zeta_{d'}) : \Q]$,
the stabiliser of $\zeta_{\dB}$ in the Galois group $\Gal(\Q(\zeta_{\dB\dB'}):\Q)$ acts transitively on the roots 
$\zeta_{\dB'}, \ldots, \zeta_{\dB'}^{\dB'-1}$.
By the hypothesis in~(i) there exists $(j,0) \in B_k$.
For each $r'$ such that $1 \le r' < \dB'$ there exists
$\sigma' \in \Gal(\Q(\zeta_{\dB\dB'}):\Q)$ such that
$\zeta_{\dB}^{\s'} = \zeta_{\dB}$ and $\smash{\zeta_{\dB'}^{\s'} = \zeta_{\dB'}^{r'}}$.
Since $\smash{v_{(j,0)}^{\s'} = v_{(j,0)}}$ and $v_{(0,1)}^{\s'} = v_{(0,r')}$, we see
that if $B_k$ meets $\{(0,1),\ldots,(0,\dB'-1)\}$ then it contains this set; a similar argument,
taking $\s \in  \Gal(\Q(\zeta_{\dB\dB'}): \Q)$ such that
 $\zeta_\dB^{\s} = \zeta_\dB^r$ and $\zeta_{\dB'}^{\s} = \zeta_{\dB'}$
now shows that
$B_k = \{0,\ldots,\dB-1\} \times \{0,\ldots,\dB'-1\}
\backslash \{(0,0)\}$, and so $G$ is $2$-transitive, contrary to assumption. Therefore (i) holds.

Having proved (i), we instead
follow Burnside's argument for (i) and~(ii). Burnside chooses~$\mathcal{O}$ 
to contain $(1,1)$ and takes $(m,0) \in B_k$.
By~\eqref{eq:Burnside1} and~\eqref{eq:Burnside2}, 
$ |\mathcal{O}| \sum_{j'=0}^{\dB'-1} c_{j'}\zeta_{\dB'}^{j'} 
= |B_k| 
\sum_{(i,i') \in \mathcal{O}} \zeta_\dB^{i m}$,
where $\smash{c_{j'} = \sum_{j : (j,j') \in B_k} \zeta_{\dB}^{j}}$ for $j' \in \{0,1,\ldots,\dB'-1\}$.
According to Burnside, this implies that the coefficients
$c_{j'}$ are constant for all $j'$. 
It appears that Burnside assumes that every rational relation between the powers of $\zeta_{\dB'}$
is a multiple of $1+ \zeta_{\dB'} + \cdots + \zeta_{\dB'}^{\dB'-1}$. But  
a more general relation is
$a + \zeta_{\dB'} + \cdots + \zeta_{\dB'}^{\dB'-1} = a-1$, 
so we can only conclude that
the $c_{j'}$ are constant for $j' \in \{1,\ldots, \dB'-1\}$. However, it \emph{is} true that
if $\sum_{j \in J} \zeta_{\dB}^{j} = \sum_{j \in K} \zeta_{\dB}^{j}$ 
for non-empty sets $J$, $K \subseteq \{0,1,\ldots, \dB-1\}$ then $J = K$, so this weaker
conclusion implies that, for each $j' \in \{1,\ldots,\dB'-1\}$, either $\{j : (j,j') \in B_k \} \supseteq \{1,\ldots, \dB-1\}$
or $\{j : (j,j') \in B_k \} \subseteq \{0\}$. Hence
\begin{itemize}
\item[(i)${}'$] If $B_k$ meets $\{(1,0),\ldots,(\dB-1,0)\}$ then
$B_k$ is a union of sets of the form $\{(j,0)\}$ and 
$\{(j,1), \ldots, (j,\dB'-1)\}$ where $1 \le j < \dB$.
\end{itemize}
The Galois action of the automorphisms $\sigma$ in our proof of (i)
shows that (i)${}'$ implies (ii). Therefore Burnside's argument can be corrected.

The final sentence of the paragraph we have been reading is `It is clear that the same method of proof
will apply, when the transitive Abelian subgroup has three or more independent generators'.
Taking $d=4$ in Example~\ref{ex:prim}, we see that the subgroup 
$\langle (0,1,2,3) \rangle \times \langle (0,1)(2,3), (0,2)(1,3)\rangle  \le G$
acts regularly in the primitive action of $G$ on $\{0,1,2,3\}^2$. 
Therefore $C_4 \times C_2 \times C_2$ is not a B-group and Burnside's claim is false.
The use of the Galois action in
the previous paragraph required that both $\dB$ and $\dB'$ are prime.

In  \S\ref{sec:BurnsideEven} below we
extend the correct part of Burnside's 
proof to show that if $p$ is an odd prime
and $n \in \N$ then $C_{2^n}$, $C_{2^n p}$ and $C_{2p^n}$ are $B$-groups.
A proof of Conjecture~\ref{conj:coprime} will
rehabilitate Burnside's method for cyclic groups.

\subsection{Manning's 1936 paper}\label{subsec:Manning}
In \cite{DManning1936}, D.~Manning claimed a proof, using Burnside's method,
that if $p$ is prime and $a > b$
then $C_{p^a} \times C_{p^b}$ is a B-group. 
It is reported in \cite[page~67]{WielandtBook1964} that she later acknowledged that
the critical Lemma II in \cite{DManning1936} is false. We extend Example~\ref{ex:prim} to show this.

\begin{example}\label{ex:Manning}
Recall from Example~\ref{ex:prim} that $S$ is the symmetric group on $\{0,1,\ldots, \dB-1\}$ and
 $G \cong S \wr C_2$ acting
primitively on $\{0,1,\ldots, \dB-1\}^2$. We took $g_\dB = g'_\dB = (0,1,\ldots, \dB-1)$.
By Example~\ref{ex:prim}, the natural $\C G$-permutation module 
has a summand with 
basis set $B = \{ (j,0), (0,j')  : 1 \le j < \dB, 1 \le j' < \dB \}$, with
respect to the chosen generators $(g_\dB,1)$ and $(1,g'_\dB)$ of the
regular subgroup~$K =  \langle g_\dB \rangle \times \langle g'_\dB \rangle$.

We have
\begin{align*} v_{(j,0)} (g_\dB,1) &= \zeta^j v_{(j,0)}, \quad v_{(0,j')} (g_\dB,1) = v_{(0,j')}, \\
v_{(j,0)} (g_\dB,g'_\dB) &= \zeta^j v_{(j,0)}, \quad v_{(0,j')} (g_\dB,g'_\dB) = \zeta^{j'} v_{(0,j')}.
\end{align*}
Therefore, with respect to the alternative generators $(g_\dB,1)$ and $(g_\dB,g'_\dB)$ of~$K$, the basis set 
becomes
\hbox{$C = \{ (j,j) : 1 \le j < \dB \} \hskip0.5pt\cup\hskip0.5pt  \{ (0,j') : 1 \le j' < \dB \}$}.
Observe that, as it must be, 
$C$ is invariant under the action induced by \hbox{$\Gal(\Q(\zeta_\dB) : \Q)$}.
Manning's Lemma II asserts the stronger property that, given any $(i,i') \in \{0,1,\ldots, \dB-1\}^2$
with $i$ and $i'$ coprime to $\dB$,  $C$ is invariant under the permutation
$(j,j') \mapsto (ij,i'j')$, where the entries are taken modulo~$\dB$. Taking $i = 1$ and $i'=-1$
we see that this is false whenever $\dB > 2$.
\end{example}

\subsection{Later proofs of Burnside's and Manning's claims}
In 1908, Schur  introduced his method of S-rings and gave the first
correct proof of Theorem~1 \cite{Schur1908}. 
In 1933 Schur extended his method to prove, more generally, that
any cyclic group of composite order is a B-group. As remarked in \cite{DManning1936}, 
it appears that Schur was unaware of Burnside's 1921 paper.
In 1935,
Wielandt wrote 
`Der von Herrn Schur angegebene Beweis ist recht schwerig',
and gave a short proof of the still more general result that 
any abelian group of composite order 
having a cyclic Sylow $p$-subgroup for some prime $p$ is a B-group \cite{Wielandt1936}.
Wielandt's proof depends on several results on S-rings, in particular property (6) in \cite{Wielandt1936},
that the stabiliser of an element of an S-ring is itself in the ring.
Wielandt's result
 and proof appear, in translation but essentially unchanged, in his 1964 textbook \cite[Theorem~25.4]{WielandtBook1964}.
The use of complex conjugation
at the end of the proof of Theorem~\ref{thm:mainAlt}  in~\S\ref{sec:mainAlt} involves some similar ideas to 
the proof of property (6) in \cite[Theorem~23.5]{WielandtBook1964}, 
 but the proof here is substantially shorter and more elementary. 
 
The first essentially correct proof of the result claimed by D.~Manning was given
by Kochend{\"o}rffer in 1937 using S-rings \cite{Kochendorffer}; 
Wielandt comments in \cite{WielandtBook1964} that it is `very complicated' (Bercov's translation). 
In his essay in~\cite{WielandtWerkeI}, Neumann reports that in an unpublished note D.~Manning found some slips in
\cite{Kochendorffer}, but was able to correct them.
Neumann's essay includes a proof of Theorem~\ref{thm:main} that
a reader, familiar with the prerequisites from modular representation theory and permutation groups,
will find spectacularly short and beautiful.

Apart from \cite{Knapp}, outlined in the introduction, the three papers
\hbox{\cite{Burnside1901, Burnside1921, DManning1936}}
surveyed in this section appear to exhaust the research literature on Burnside's method.
It is intriguing that all err in ultimately the same way, by overlooking algebraic relations
satisfied by roots of unity.

\subsection{Burnside's method in even degree}\label{sec:BurnsideEven}
Again we continue from the end of~\S\ref{sec:preliminaries}.
There is an action of the Galois group $\Gal(\Q(\zeta_{d}) : \Q)$
on the set $\{1,\ldots, d-1\}$ 
under which $\sigma \in \Gal(\Q(\zeta_{d}) : \Q)$ sends
$i$ to $i'$ if and only if $\sigma$ sends $\zeta^i$ to $\zeta^{i'}$.
In \cite[Theorem 2.3(2)]{Knapp}, Knapp extends Burnside's arguments to show that 
this action  induces an action of the Galois group on the orbits of the point stabiliser $H$.
 (This result may also be proved
using S-rings: see \cite[Theorem 23.9]{WielandtBook1964}.)
Let~$D$ be the set of divisors of $d$. Set $\R(1) = \{0\}$ and for $r \in D$ with $r > 1$, set
\[ \R(r) = \{m d/r : 0 < m < r, 
\hcf(m, r) = 1 \}. \]
Thus for each $r \in D$ the set $\{ \zeta_d^i : i \in \R(r)\}$, consisting of all primitive $r$-th
roots of unity, is an orbit of the Galois group on the powers of $\zeta_d$.
If $d$ is even then, since $\R(2) = \{d/2\}$ corresponds to $\zeta_d^{d/2} = -1 \in \Q$,
the $H$-orbit $\O$ containing $d/2$
 is invariant under the Galois action. Hence
$\mathcal{O} = \bigcup_{r \in \X} \R(r)$
for some subset $\X$ of $D$. Observe that $G$ is $2$-transitive
if and only if $\X = D \backslash \{1\}$. 

For $r \in D$ and $j \in \N$ we have
$\sum_{i \in \R(r)} \zeta_{d}^{ij} = \sum_\alpha \alpha^j$
where $\alpha$ ranges over all primitive $r$-th roots of unity.
If $\hcf(r,j) = j^\star$ then the map
$\alpha \mapsto \alpha^j$ is $j^\star$ to~$1$, and each $\alpha^j$
is a primitive $r/j^\star$-th root of unity.
It is well known that the sum of the $\phi(s)$ roots of unity of order $s$ is $\mu(s)$,
 where~$\mu$ is the M{\"o}bius function 
(see for instance \cite[Exercise 2.8]{WashingtonCyclotomic}).
Therefore, if~$R$ is the matrix with rows and columns
indexed by $D$, defined by
\begin{equation}
\label{eq:Rsum} R_{rc} = \mu\Bigl( \frac{r}{\hcf(r,c)}\Bigr) \frac{\phi(r)}{\phi(\frac{r}{\hcf(r,c)})}
\end{equation}
then, for any $r \in D$ and $j \in \N$,
\begin{equation}
\label{eq:Rsums} 
\sum_{i \in \R(r)} \zeta_d^{ij} = R_{rc} \quad\text{where $c = \hcf(d,j)$.} 
\end{equation}
(Here $R$ stands for Ramanujan, who considered these
cyclotomic sums in~\cite{Ramanujan1918}; this was published in the interval between Burnside's 1901 and 1921
papers,
but there is no evidence that Burnside was aware of its relevance.)
As an \emph{aide-memoire}, 
we note that $R_{rc}$ is defined by taking $c$-th powers of
$r$-th roots of unity. An example of these matrices is given after Lemma~\ref{lemma:pMatrix}.

Let $\sim_\X$ be the relation on $D \backslash \{d\}$ defined by
\begin{equation}\label{eq:sim} b \sim_\X c \iff \sum_{r \in \X} R_{rb} = \sum_{r \in \X} R_{rc}. \end{equation}
Let $\P_\X$ be the set of equivalence classes of $\sim_\X$.
Given  $B \subseteq \{0,1,\ldots,d-1\}$, 
let $Y(B) = \{ c \in D \backslash \{d\} : B \cap \R(d/c) \not= \varnothing \}$.
For example, $1 \in Y(B)$ if and only if $B$ contains a number coprime to $d$,
and \hbox{$Y(\{0\}) = \varnothing$}.
If $B_k$ and $B_\ell$ are distinct basis sets then necessarily
$B_k \hskip1pt\cap\hskip1pt B_\ell = \varnothing$, but if neither $B_k$ nor $B_\ell$ is invariant under the Galois action,
we may still have $Y(B_k) \cap Y(B_\ell) \not= \varnothing$. However the asymmetry
between orbits and basis sets in the conclusion of Lemma~\ref{lemma:firstStep} works
in our favour, to show that  
$\sum_{r \in \X} R_{rc}$
is constant for $c \in Y(B_k)$. 
It follows that  $Y(B_k)$
is contained in a single part of the partition $\P_\X$ of $D \backslash \{d\}$.
Hence, by Proposition~\ref{prop:imprimitive}, we may assume that
the highest common factor of the entries in each part of the partition $\P_\X$ of
$D \backslash \{d\}$ is $1$.
We say that such partitions are \emph{coprime}.

For $c \in D$,  an easy calculation from~\eqref{eq:Rsums} shows that

\vspace*{-13pt}
\[ \sum_{r \in D} R_{rc} = 
 \sum_{i=0}^{d-1} \zeta_d^{ic} = c\sum_{i=0}^{d/c-1} \zeta_{d/c}^i = \begin{cases}
0 & \text{if $c < d$,} \\ d & \text{if $c = d$.} \end{cases} \]
Since $R_{1c} = 1$ for all $c \in D$, it follows
that if $\X = D \backslash \{1\}$ then $\P_\X = \bigl\{ D \backslash \{d\} \bigr\}$. 
This proves the `only if' direction of the following conjecture.

\begin{conjecture}\label{conj:coprime}
Let $\X \subseteq D$ contain $2$.  The partition $\P_\X$ of $D \backslash \{d\}$ defined by
the relation $\sim_\X$ in~\eqref{eq:sim} is coprime
if and only if $\X = D \backslash \{1\}$ or $\X=D$.
\end{conjecture}

We have shown that if $d$ is even then,
defining $\X$ as above by the orbit~$\mathcal{O}$ containing
$d/2$, the `if' direction of Conjecture~\ref{conj:coprime}
implies that $\X = D \backslash \{1\}$ and $\mathcal{O} = \{1,\ldots, d-1\}$, and so
$C_d$ is a $B$-group.

The following lemma can be used to prove Conjecture~\ref{conj:coprime} in several
 cases of interest. Let $R(d)$ denote the Ramanujan matrix defined for degree $d$.


\begin{lemma}{\ }\label{lemma:pMatrix}
\begin{thmlist}
\item Let $p$ be prime and let $n \in \N$. We have 
\[ R(p^n)_{p^ep^f} = \begin{cases} 0 & \text{if $f < e-1$,} \\
-p^{e-1} & \text{if $f = e-1$,} \\
(p-1)p^{e-1} & \text{if $f \ge e$.} \end{cases}\]
\item Let $p_1,\ldots, p_s$ be distinct primes and let $n_1, \ldots, n_s \in \N$. We have
$R(d) = R(p_1^{n_1}) \otimes \cdots \otimes R(p_s^{n_s})$.
\end{thmlist}
\end{lemma}

\begin{proof}
Part (i) is immediate from~\eqref{eq:Rsum}. For (ii), it suffices to show
that if $d$ and $d'$ are coprime and $r \mid d$, $r' \mid d'$ and $c \mid d$, $c' \mid d'$ then
the entry in row $rr'$ and column $cc'$ of $R(dd')$ is
$R_{rc}(d)R_{r'c'}(d')$. This follows from~\eqref{eq:Rsum} using the multiplicativity
of $\mu$ and $\phi$, noting that $\hcf(r,r') = \hcf(c,c') = 1$.
\end{proof}

For example, if $p$ is an odd prime then $R(2p^3)$ is as shown below,
with~$D$ ordered $1,3,9,27,2,6,18,54$ and row $2 \in \X$ highlighted.
The division indicates the tensor factorization $R(p^3) \otimes R(2)$. 
\[ \scalebox{0.8}{$\displaystyle \left( \begin{matrix}[cccc|cccc] 1 & 1 & 1 & \mbf{1}  & 1 & 1 & 1 & 1 \\
-1 & p-1 & p-1 & \mbf{p-1} & -1 & p-1 & p-1 & p-1 \\
0  & -p  & p(p-1) & \mbf{p(p-1)} & 0 & -p & p(p-1) & p(p-1) \\
0 & 0 & -p^2 & p^2(p-1) & 0 & 0 & -p^2 & p^2(p-1)\\ \hline
\mbfr{-1} & \mbfr{-1} & \mbfr{-1} & \mbfr{-1}  & \mbfr{1} & \mbfr{1} & \mbfr{1} & \mbfr{1} \\
1 & -(p-1) & -(p-1) & \mbf{-(p-1)} & -1 & p-1 & p-1 & p-1 \\
0  & p  & -p(p-1) & \mbf{-p(p-1)} & 0 & -p & p(p-1) & p(p-1) \\
0 & 0 & p^2 & \mbf{-p^2(p-1)} & 0 & 0 & -p^2 & p^2(p-1)
\end{matrix} \right)$} \]
In particular $R(p^3)$ appears as the top-left block.

\begin{proposition}\label{prop:conj}
Let $n \in \N$ and let $p$ be an odd prime.
Conjecture~\ref{conj:coprime} holds when 
\emph{(i)} $d=2^n$, \emph{(ii)} $d= 2^n p$ and \emph{(iii)} $d=2p^n$.
\end{proposition}

\begin{proof}
The `if' direction remains to be proved. Recall that the
 rows and columns of $R$ are labelled by the divisors of $d$.
 Since row $1$ of $R(d)$
is constant, we may assume that $1 \in \X$. 

Suppose, as in (i), that $d = 2^n$. If $n=1$ then $\X = \{1,2\}$ and the conclusion is immediate.
Suppose that $n \ge 2$. 
Let~$R^\star$ be the matrix obtained from
$R(2^n)$ by deleting row
$1$ and replacing row $2$ with the sum of rows $1$ and~$2$.
Observe that column $1$ of $R^\star$ has all zero entries,
and the submatrix of~$R^\star$ formed by
columns $2^f$ for $1 \le f \le n$
is $2R(2^{n-1})$. Therefore $\sum_{r \in \X} R(2^n)_{rc} = \mfrac{1}{2} \sum_{r \in \X^\star} R(2^{n-1})_{rc}$
where $\X^\star = \{1\} \cup \{r/2 : r \in \X \backslash \{1,2\} \bigr\}$. 
By induction $\X^\star = \{1,2,\ldots, 2^{n-1} \}$, and so $\X = D$.

Part (ii) follows by a small extension of this argument.
Let $R^\star$ be as defined in (i). By Lemma~\ref{lemma:pMatrix}, the entry of $R^\star$ in row $r$ and column $c$ is
odd if and only if $r \in \{p,2p\}$ and $c = 2^m$ where $0 \le m \le n$.
Any coprime partition has a part
containing both $2^m$ and $p$ for some such $m$. Therefore, by parity, either both $p$ and $2p$ are
contained in $\X$, or neither are. Deleting row $p$ and replacing row $2p$ with the sum
of rows $p$ and $2p$ of $R^\star$, we obtain $2R(2^{n-1}p)$,
augmented by two zero columns. The inductive argument for~(i) now shows that $\X = D$.

Finally suppose that $d=2p^n$. Let $\overline{R}(2p^n)$ denote $R(2p^n)$ with
entries regarded as elements of $\Z/p^n\Z$.
Let $\simmod$ be the relation on $D\backslash \{2p^n\}$
defined as in~\eqref{eq:sim}, but working modulo $p^{n}$. Let $\overline{\mathcal{P}}_\X$
denote the set of equivalence classes for $\simmod$. We need this preliminary result: 
\emph{if $\overline{\mathcal{P}}_\X$ is
coprime then $2$, $2p$, \ldots, $2p^{n} \in \X$ and $\overline{\mathcal{P}}_\X$ has a single part.}
Again the proof is inductive. If $n = 1$ then, by
Lemma~\ref{lemma:pMatrix},
\[ \overline{R(2p)} = \left( \begin{matrix}[cc|cc] 1 & 1 & 1 & 1 \\
-1 & -1 & -1 & -1 \\ \hline
-1 & -1 & 1 & 1 \\
1  &  1 & -1 & -1
\end{matrix} \right) \]
where the entries are in $\Z/p\Z$ and~$D$ is ordered $1,p,2,2p$.
If $2p \not\in \X$ then, since $1$, $2 \in \X$, we have $\overline{\mathcal{P}}_\X = \bigl\{ \{1,p\}, \{2,2p\} \bigr\}$,
which is not coprime. Therefore $2p \in \X$ and $\overline{\mathcal{P}}_\X = \bigl\{ \{1,p,2,2p\} \bigr\}$, 
as required. Suppose that $n \ge 2$. Let $\overline{R}^\hstar$ denote
$R(2p^n)$ with the entries taken in $\Z/p^{n-1}\Z$. Observe that
columns $p^{n-1}$ and $p^n$ of $\overline{R}^\hstar$ are equal, as are columns $2p^{n-1}$
and $2p^n$. Moreover, rows $p^n$ and $2p^n$ have all  zero entries. 
By a very similar
inductive argument to (i), it follows that $\X$ contains
$2$, $2p$,\ldots, $2p^{n-1}$. Let $R^\star$ be the matrix
obtained from $R(p^n)$ by removing these rows, replacing row $2$ with their sum,
and adding $p^{e-1}$ to each entry in row $p^e$, for $1 \le e \le n$.
For example, if $n=3$ then
\[ R^\star = \scalebox{0.8}{$\displaystyle 
\left( \begin{matrix}[cccc|cccc] 1 & 1 & 1 & \mbf{1}  & 1 & 1 & 1 & 1 \\
0 & p & p & p & 0 & p & p & p \\
p  & 0  & p^2 & p^2 & p & 0 & p^2 & p^2 \\
p^2 & p^2 & 0 & p^3 & p^2 & p^2 & 0 & p^3\\ \hline
\mbfr{0} & \mbfr{0} & \mbfr{-p^2} & \mbfr{-p^2}  & \mbfr{0} & \mbfr{0} & \mbfr{p^2} & \mbfr{p^2} \\
0 & 0 & p^2 & \mbf{-p^2(p-1)} & 0 & 0 & -p^2 & p^2(p-1)
\end{matrix} \right)$} \]
where the row obtained by summation is highlighted. Since columns $1$ and~$2$ of $R^\star$ are equal, and any part
of a coprime partition of $D \backslash \{2p^n\}$ contains either $1$ or~$2$, we see that~$\overline{\mathcal{P}}_\X$ has
a single part. The column of $R^\star$ labelled $2p^{n-1}$ is
greater, entry-by-entry, than every other column, except in  rows $p^n$ and $2p^n$.
Since columns $p^{n-1}$ and $2p^{n-1}$ of $R^\star$ are congruent except in the summed row and row $2p^n$,
and the sum of entries in these columns is less than~$p^n$,
we have $2p^n \in \X$. This proves the preliminary result. 
%

We now prove (ii). Each part of $\overline{\mathcal{P}}_\X$ is a union of parts of $\mathcal{P}_\X$,
so $\overline{\mathcal{P}}_\X$ is coprime only if $\overline{\mathcal{P}}_\X$ is coprime.
By the preliminary result, $2$, $2p$, \ldots, $2p^{n} \in \X$.
Let $R^{\star\star}$ be the matrix defined as $R^\star$, but now adding all
the rows $2$, $2p$, \ldots, $2p^{n-1}$, $2p^n$. The non-zero entries
in the summed row for $R^{\star\star}$ are $-p^{n}$ in column $p^n$
and $p^{n}$ in column $2p^n$. Since $p^{n}$ is in a non-singleton part of $\mathcal{P}_\X$,
we see from column $p^n$ that
$\X$ contains $1, p, \ldots, p^n$, as required.
\end{proof}

Despite its elementary statement, 
the author has been unable to prove Conjecture~\ref{conj:coprime}
in any significantly greater generality. We offer this as an open problem.

The {\sc Haskell} \cite{Haskell98} program \texttt{RamanujanMatrix} on
the author's website\footnote{See \url{www.ma.rhul.ac.uk/~uvah099/}} has been used to verify
Conjecture~\ref{conj:coprime} for all 
degrees $d \le 600$. 
We mention that
\[ R(p^n) = \scalebox{0.8}{$\displaystyle \left( \begin{matrix} 1 & 0 & 0 & \ldots & 0 \\ 1 & 1 & 0 & \ldots & 0 \\ 1 & 1 & 1 & \ldots & 0 \\
\vdots & \vdots & \vdots & \ddots & \vdots \\
1 & 1 & 1 & \ldots & 1 \end{matrix} \right)^{-1}
\left( \begin{matrix} 1 & 1 & 1 & \ldots & 1 \\ 
0 & p & p & \ldots & p \\
0 & 0 & p^2 & \ldots & p^2 \\
\vdots & \vdots & \vdots & \ddots & \vdots \\
0 & 0 & 0 & \ldots & p^n \end{matrix} \right)$}. \]
It follows that each $R(d)$ is invertible;
the determinant of $R(p^n)$ is
$p^{n(n+1)/2}$ and its inverse is
$ R(p^n)^\circ /p^n$ where $R(p^n)^\circ$ is 
obtained from $R(p^n)$ by rotation by a half-turn. This leads
to an alternative proof of Proposition~\ref{prop:conj}(i) 
and may be useful more widely.

\section{Abelian B-groups}
\label{sec:abelian}

\subsection{After CFSG}
We now skip over many later developments, referring the reader to Neumann's essays in 
the collected works
\cite{BurnsideCWorksI, WielandtWerkeI}
for some of the missing history, 
and consider the situation after the Classification
Theorem of Finite Simple Groups. In an early application, it was used in \cite{CurtisKantorSeitz}
to determine all $2$-transitive permutation groups.
The resulting classification of all primitive permutation groups containing a regular
cyclic subgroup is given in \cite[Theorem 4.1]{FeitSomeConsequences} and 
\cite[page 164]{KantorSomeConsequences}, and independently refined in \cite{JonesCyclicRegular} and \cite{LiCyclicRegular}. We
state the version of this result relevant to Theorem~\ref{thm:main}~below.
(Here $S_d$ and $A_d$ denote the symmetric and alternating groups of degree~$d$, respectively;
the other notation is also standard.)

\begin{theorem}\label{thm:cyclic}
Let $G$ be a permutation group containing a regular cyclic subgroup $\langle g \rangle$ 
of composite prime-power order $p^n$. Then either $G$ is imprimitive, or $G$ is $2$-transitive and
one of the following holds:
\begin{itemize}
\item[(i)] $G = A_{p^n}$ or $G = S_{p^n}$ and $g$ is a $p^n$-cycle;
\item[(ii)] $\PGL_d(\F_q) \le G \le \PGammaL_d(\F_q)$ where $p^n = (q^d-1)/(q-1)$;
\item[(iii)] $p=3$, $n=2$, $G = \PGammaL_2(\F_8)$ and $g = s \sigma$ where $s \in \PGL_2(\F_8)$ is semisimple
of order $3$ and $\sigma$ is the automorphism of $\PGL_2(\F_8)$ induced by the Frobenius twist.
\end{itemize}
\end{theorem}

Corollary~3 of \cite{LiebeckPraegerSaxl2000} gives a rough classification of primitive
permutation groups containing a regular subgroup. This was sharpened  by Li in 
\cite[Theorem~1.1]{LiAbelianRegular} for regular abelian subgroups.
Note that Case (2)(iv) of this theorem,
on groups with socle $S_m \times \cdots \times S_m$
or $A_m \times \cdots \times A_m$, is missing the assumption $m \ge 5$.
It is clear from Remark (b) following the theorem and the structure
of the proof in \S 5 that this assumption was intended;
it is required to exclude groups such as $S_2 \wr S_r$
and $A_3 \wr S_r$ with regular socle whose product action is imprimitive. 
(Primitive groups such
as~$S_4$ in its natural action or 
$S_3 \wr S_2$ in its product action are of affine type, and so
already considered in Case (1) of the theorem.)

It will
be useful to 
say that a group $K$ is \emph{$m$-factorizable} if there exists $r \ge 2$
and groups $K_1, \ldots, K_r$ such that $|K_1| = \ldots = |K_r| = m$
and $K \cong K_1 \times \cdots \times K_r$, and \emph{factorizable}
if it is $m$-factorizable for some $m \ge 3$.

\begin{proposition}\label{prop:affineOrFact}
If $K$ is a regular abelian subgroup of a primitive but not $2$-transitive 
permutation group $G$ then \emph{either}
\begin{itemize}
\item[(i)] $G = V \rtimes H$ where $V \cong \F_p^n$ is elementary abelian, the point
stabiliser $H \le \GL(V)$ acts irreducibly on $V$ but intransitively on 
$V \backslash \{0\}$ and $|K| = p^n$;~\emph{or}
\item[(ii)]  $K$ is $m$-factorizable for some $m \ge 5$. 
\end{itemize}
\end{proposition}

\begin{proof}
If Case (1) of Li's theorem applies then $G \le \mathrm{AGL}_d(\F_p)$
where~$p$ is prime and $G$ acts on its socle $V \cong \F_p^d$. It is
easy to show (see for example \cite[Theorem 4.8]{DixonMortimer}) that $G = V \rtimes H$
where $H \le \GL(V)$ is irreducible. Since $G$ is not $2$-transitive,
$H$ is not transitive. In the remaining case of Li's theorem,~$G$ 
is of the form $(\widetilde{T}_1 \times \cdots \times \widetilde{T}_r) . O . P$
where $O \le \Out(\widetilde{T}_1) \times \cdots \times \Out(\widetilde{T}_r)$, $P$ is 
transitive of degree $r$ and each $\widetilde{T}_r$ is an
almost simple permutation group of degree $m \ge 5$.
Moreover $K = K_1 \times \cdots \times K_r$ where 
$K_i < \widetilde{T}_i$ and each $K_i$ has order~$m$. Therefore,
if $r \ge 2$, then $K$ is factorizable into $m$-subgroups with $m \ge 5$.
If $r=1$ then, as Li remarks following his theorem, $G$ is $2$-transitive,
so need not be considered any further.
\end{proof}

Theorem~25.7 in \cite{WielandtBook1964} generalizes Example~\ref{ex:prim} to show that
if $m \ge 3$ and~$K$ is $m$-factorizable with $r$ factors then $K$ is a regular subgroup of $S_m \wr S_r$
in its primitive action on $\{1,\ldots,k\}^r$. This action is not $2$-transitive,
 so $K$ is not a B-group. We therefore
have the following corollary, first observed in \cite[Corollary~1.3]{LiAbelianRegular}.

\begin{corollary}\label{cor:Bfact}
No factorizable group is a B-group. Moreover,
an abelian group not of prime-power order is a B-group
if and only if it is not factorizable. 
\end{corollary}

It is an open problem to determine the non-factorizable abelian
B-groups of prime-power order. We end 
with some partial results and reductions.


\subsection{Elementary abelian B-groups}
Exercise 3.5.6 in \cite{DixonMortimer} asks for a proof that $C_p^n$ is never a B-group.
This is true when  $p > 2$ by Corollary~\ref{cor:Bfact} (clearly~$C_p$ in its regular
action is primitive but not $2$-transitive), but
false, in general, when $p=2$.\footnote{This mistake is corrected in the errata available
at
\url{people.math.carleton.ca/~jdixon/Errata.pdf}.} 
For example, the primitive permutation groups of 
degree $8$ containing a regular
subgroup isomorphic to~$C_2^3$ are
$A_8$, $S_8$ and the affine groups $\F_2^3 \rtimes C_7$,
$\F_2^3 \rtimes (C_7 \rtimes C_3)$ and $\F_2^3 \rtimes \GL_3(\F_2)$. All of 
these groups contain a $7$-cycle, and so are $2$-transitive. Therefore $C_2^3$ is a B-group. 

These examples motivate the following lemma, whose proof requires
Burnside's dichotomy 
on permutation groups of prime degree. The significance of Mersenne
primes will be seen shortly.

\begin{lemma}\label{lemma:CameronKantor}
Let $V = \F_2^n$ where $2^n-1$ is prime. A subgroup $H \le \GL(V)$ is transitive on $V \backslash \{0\}$
if and only if $H \cong C_{2^n-1}$, $H \cong C_{2^n-1} \rtimes C_n$ or $H = \GL(V)$.
\end{lemma}

\begin{proof}
The `if' direction is clear.
By Theorem~\ref{thm:Burnside1901}, if $H$ is transitive on $V \backslash \{0\}$ then either $H \cong C_{2^n-1}
\rtimes C_r$, for some $r$, or $H$ is $2$-transitive. 
Identifying $V \backslash \{0\}$ with $\F_{2^n}^\times$, 
we see that there exists $h \in \GL(V)$ of order $2^n-1$. (Such elements are called Singer cycles.)
Let $\alpha$ be a primitive $(2^n-1)$-th root of unity. Note that $h$ is conjugate to $h^s$ in $\GL(V)$ if and only
if the map $\beta \mapsto \beta^s$ permutes the eigenvalues $\alpha, \alpha^2, \ldots, \alpha^{2^{n-1}}$ of $h$. Thus $N_{\GL(V)}(\langle h \rangle) \cong C_n$
is generated by an element of prime order $n$ conjugating~$h$ to $h^2$, and either $r=1$ or $r=n$.
If $H$ is $2$-transitive then $V \rtimes H$ is $3$-transitive. 
Such groups were classified by Cameron and Kantor in \cite{CameronKantor}.
By their Theorem 1 in the case of vector spaces over $\F_2$,
the only such groups are $V \rtimes \GL(V)$ and, when $n=4$, $V \rtimes A_7$.
Since $2^4-1$ is composite,
only the former case arises.
\end{proof}

It is worth noting that \cite{CameronKantor} predates the classification theorem; the methods
used are mainly from discrete geometry rather than group theory. More generally, Hering \cite{HeringI, HeringII}
has classified the linear groups $H$ transitive on non-zero vectors, under various assumptions on
the composition factors of~$H$.


\begin{proposition}\label{prop:Mersenne}
Let $V = \F_2^n$. The elementary abelian group $C_2^n$ is a B-group if and only if $2^n-1$ is a Mersenne prime
and the only simple irreducible subgroups of $\GL(V)$ are $C_{2^n-1}$ and 
$\GL(V)$.
\end{proposition}

\begin{proof}
Suppose that $2^n-1$ is composite. Let $h \in \GL(V)$ be a Singer cycle.
If $n\not=6$ then, by Zsigmondy's Theorem \cite{Zsigmondy}, there exist a prime $r$ such that~$r$ divides $2^n-1$ and $r$
does not divide $2^m-1$ for any $m < n$. Thus $h^{n/r}$ does not permute the vectors of a non-zero proper subspace
of $V$, and so  $\langle h^{n/r} \rangle$ acts irreducibly on $V$
and intransitively on $V \backslash \{0\}$. Therefore $V \rtimes \langle h^{n/r} \rangle$ is primitive
but not $2$-transitive, and so $C_2^n$ is not a B-group. In the exceptional case of Zsigmondy's Theorem when $n=6$,
we simply take~$h^3$, of order~$21$.

Suppose that $2^n-1$ is prime and that there is a simple irreducible group $T \le \GL(V)$
other than $C_{2^n-1}$ and $\GL(V)$. By Lemma~\ref{lemma:CameronKantor}, $T$ is intransitive on $V \backslash \{0\}$,
and so $V \rtimes T$ is not $2$-transitive. Hence $C_2^n$ is not a B-group. Conversely, assume that
no such simple group exists, and, for a contradiction, that $C_2^n$ is not a B-group. 
By Proposition~\ref{prop:affineOrFact},
there exists a proper irreducible subgroup $H$ of $\GL(V)$ such that $H$ is intransitive on $V \backslash \{0\}$.
Let~$M$ be a maximal subgroup of $\GL(V)$ containing $H$. The maximal subgroups of classical groups
were classified by Aschbacher in \cite{AschbacherMaximals}. Of the~$11$ Aschbacher classes, the first consists
of reducible groups, and the remaining~$10$ of groups preserving a structure on $V$ that can exist only when $V$
has composite dimension. Therefore $M$ is an almost simple group. Since $M$ is a proper subgroup of $\GL(V)$,
Lemma~\ref{lemma:CameronKantor} implies that $M$ is intransitive on $V \backslash \{0\}$.
Let $T$ be the simple normal subgroup of $M$. By Clifford's Theorem (\cite[Theorem~I]{Clifford}), the restriction
of $V$ to $T$ decomposes as a direct sum of irreducible representations of $T$ of the same dimension. 
Since $n$ is prime, $T$ acts irreducibly on $V$. Its
orbits are contained in the orbits of $M$, so it acts intransitively on $V \backslash \{0\}$,
contrary to our assumption.
\end{proof}

By Proposition~\ref{prop:Mersenne},
a solution to the following problem will  imply that~$C_2^n$ is a B-group
if and only if $2^n-1$ is a Mersenne prime. 

\begin{problem}\label{prob:simpleOdd}
Show that if $2^n-1$ is a Mersenne prime and $n \ge 3$ then
no non-abelian finite simple group other than $\GL_n(\F_2)$ has an 
irreducible representation of dimension~$n$ over~$\F_2$.
\end{problem}

The two remarks below give some partial progress on Problem~\ref{prob:simpleOdd}.
\begin{itemize}
\item[(1)]The Atlas \cite{AtlasBrauer} data available in {\sc gap} \cite{GAP4}
shows that, with the possible exceptions of $J_4$, $Ly$, $Th$, $Fi_{24}$, $B$ and $M$, 
no sporadic simple group has an irreducible representation
over $\F_2$ of dimension $n$ where $2^n-1$ is a Mersenne prime.
Indeed, it appears to be rare for a sporadic group or a finite group 
of Lie type to have a non-trivial irreducible representation over~$\F_2$
of odd dimension. The author knows of no examples of such representations
of alternating groups. Since a self-dual representation has an invariant
alternating form, whereas an odd-dimensional orthogonal group over $\F_2$
has a $1$-dimensional invariant subspace, such a representation is necessarily
not self-dual.

\smallskip
\item[(2)] Inspection of tables of small dimensional representations of quasisimple groups 
\cite{HissMalle, HissMalleCorrigenda} and (for the groups deliberately excluded therein),
Chevalley groups in defining characteristic
\cite{Lubeck} show that no finite simple group except for $\GL_n(\F_2)$ 
has an irreducible representation over $\F_2$ of dimension~$n \le 250$
such that $2^n-1$ is a Mersenne prime.
Thus if $n \le 250$ then $C_2^n$ is a B-group if and only if
 \[ \qquad n \in \{1,2,3,5,7,13,17,19,31,61,89,107,127\}. \]
\end{itemize}

\subsection{Non-elementary abelian $B$-groups}

An interesting feature of the affine groups in Proposition~\ref{prop:affineOrFact}(i) is that
they may contain regular abelian subgroups other than $C_p^n$. In Remark~1.1 
in~\cite{LiAbelianRegular}, Li gives the example $V \rtimes S_n$ where $V$ is the
subrepresentation $\langle e_2-e_1, \ldots, e_n-e_1 \rangle$ of the natural
permutation representation $\langle e_1, \ldots, e_n \rangle$ of $S_n$ over $\F_p$. 
To avoid a potential ambiguity, let $t_v \in V \rtimes H$ denote translation by $v \in V$.
If $2s < n$ then the subgroup of $V \rtimes H$ generated by
\[ (2,3)t_{e_1+e_2},\, \ldots,\, (2s,2s+1)t_{e_1+e_{2s}},\, t_{e_{2s+2}},\, \ldots, t_{e_n} \]
is regular and isomorphic to $C_4^s \times C_2^{n-2s-1}$.
Li  claims that $V \rtimes H$ is primitive. However $H$ acts irreducibly only
when $n$ is odd (and so $\dim V$ is even, as expected by Remark~(1) above).
Thus if $r \in \N_0$ and $s \in \N$ then $C_4^s \times C_2^{2r}$ is not a B-group,
but Li's example sheds no light on when $C_4^s \times C_2^{2r+1}$, which may be non-factorizable, 
is a B-group. This is a special case of the following~problem.

\begin{problem}\label{prob:nonelem}
Classify non-elementary abelian B-groups of prime-power order.
\end{problem}

By Proposition~\ref{prop:affineOrFact}, this problem reduces to classifying
regular abelian subgroups of affine groups $V \rtimes \GL(V)$. The main result of
\cite{CDVS} is a beautiful bijective correspondence between such subgroups and
nilpotent algebras with underlying vector space $V$. 
To explain part of this correspondence, observe that if $K$ is an regular abelian subgroup of $V \rtimes H$ 
where $H \le \GL(V)$
then, for each $v \in V$,
there exists a unique $h_v \in H$ such that $h_v t_v \in K$. 
From
$h_u h_v t_{uh_v + v} = h_u t_u h_v t_v = h_vt_vh_ut_u=   h_vh_u t_{vh_u+u}$ 
for $u,v \in V$, we see that $\{h_v : v \in V\}$ is an abelian subgroup of $H$ 
and $uh_v + v = vh_u + u$ for all $u$, $v \in V$. Replacing~$v$ with $v+w$, we obtain
\[ uh_{v+w} + (v+w) = (v+w)h_u + u 
                 = vh_u + wh_u + u 
                 = uh_v + v + uh_w + w - u\]
and so, cancelling $v+w$ and subtracting $u$, we have
\begin{equation}\label{eq:linearity}  h_{v+w} - 1 = (h_v - 1) + (h_w - 1)  \end{equation}
for all $v$, $w \in K$. This additivity property is highly restrictive.

\begin{example}
Let $K =\{h_v t_v : v \in V\}$ 
be a regular abelian subgroup of $V \rtimes S_n$, where $V$ is as in Li's example.
The matrix $X$ representing $h_v$ in the basis $e_2-e_1$, \ldots, $e_n-e_1$ of $V$
is a permutation matrix if and only if~$1h_v = 1$. If $1h_v = a$
and $bh_v = 1$ then, since $(e_i-e_1)h_v = (e_{ih_v}-1)-(e_a-1)$, each entry
of $X$ in column $e_a-e_1$ is $-1$, row $e_i-e_1$ has a $1$
in column $e_{ih_v}-e_1$ for each $i\not=b$, and $X$ has no other non-zero entries.
By~\eqref{eq:linearity},
$h_{2v} = 2h_v - 1$, so $h_{2v}$ is represented by $2X - I$, where $I$ is the identity matrix.
But $2X - I$ is not of either of these forms unless $p=2$ or $X = I$.
Therefore $V$ is the unique regular abelian
subgroup of $V \rtimes S_n$ if $p > 2$. Suppose that $p=2$. If $h_v$ has order $4$ or more, the matrix representing
$h_v+h_v^{-1} + 1$ has multiple non-zero entries in the columns for both $e_a-e_1$ and $e_b-e_1$, again 
contradicting~\eqref{eq:linearity}.
Therefore each $h_v$ has order at most $2$. It follows that $K$ has exponent $2$ or~$4$. Thus
the examples given by Li are exhaustive.
\end{example}

When $p$ divides $n$ the representation $V$ has an
irreducible quotient $U = V / \langle e_1 + \cdots + e_n \rangle$. Similar arguments show that
$U \rtimes S_n$ has a non-elementary abelian 
regular subgroup if and only if $p=2$. Any such subgroup has exponent~$4$, with the
exception that when  $n=6$, $U \rtimes S_6$ has an regular abelian subgroup isomorphic to $C_8 \times C_2$.
This does not contradict the result first claimed by Manning (see \S\ref{subsec:Manning})
since in this case $S_6$ acts transitively on $U \backslash \{0\}$; the related
$2$-transitive action of $A_7$ on $\F_2^4$, coming from the isomorphism
$A_8 \cong \GL_4(\F_2)$, was seen in the proof of Lemma~\ref{lemma:CameronKantor}.

We end with some consequences of the following observation: if $J$ is the $m \times m$ unipotent Jordan block matrix
over $\F_p$
then $J^{p^r} = I$ if and only if $p^r \ge m$ and
$I + J + \cdots + J^{p^r-1} = 0$ if and only if $p^r > m$. (The latter can be proved most simply
using the identity $I + J + \cdots + J^{p^r-1} = (J-I)^{p^r-1}$.)

\begin{proposition}\label{prop:bound}
Let $V = \F_p^n$ and let $K$ be a regular abelian subgroup of $V \rtimes \GL(V)$.
\begin{thmlist}
\item If $n < p$ then $K \cong C_p^n$.
\item If $K \cong C_{p^n}$ then either $n=1$ or $p=2$ and $n=2$.
\end{thmlist}
\end{proposition}

\begin{proof}
For $h_v t_v \in K$ we have $ (h_v t_v)^{p^r} = h_v^{p^r} \hskip-1pt t_w$ 
where $w = v + vh_v + \cdots + vh_v^{p^r-1}$.
Hence, using the observation just made, if $n < p$ then $(h_v-1)^p = 0$ and so $h_v^p = 1$ and $(h_v t_v)^p = 1$,
giving (i). Now suppose that $h_v t_v$ generates~$K$. Since $(h_vt_v)^{p^{n-1}} \not= 1$, we have
$v + vh_v + \cdots + vh_v^{p^{n-1}-1} \not=0$.
Hence there is a $m \times m$ unipotent Jordan block in $h_v$ with $m \ge p^{n-1}$. Therefore $n \ge p^{n-1}$
which implies (ii).
\end{proof}

The subgroups $K$ in Proposition~\ref{prop:bound}(i)
may be classified up to conjugacy in the affine group using the theory in \cite{CDVS}.
Using Proposition~\ref{prop:bound}(i) to rule out degrees, it follows from an exhaustive search
through the library of primitive permutation groups in {\sc magma} \cite{Magma} that the
abelian B-groups of composite 
prime-power degree $d$ where $d \le 255$ are precisely those listed in Table 1 above.
Finally we 
remark that Proposition~\ref{prop:affineOrFact}  and Proposition~\ref{prop:bound}(ii) together
imply that $C_{p^n}$ is a B-group for all primes~$p$ and all $n \in \N$ with $n \ge 2$, giving
one final proof of Theorem~\ref{thm:main}.

\begin{table}
\begin{center}
\begin{tabular}{llll} \toprule
	$d$ & $p^n$ & $f(d)$ &  abelian B-groups of order $d$ \\ \midrule
	4   & $2^2$ & 0 & $C_2^2$, $C_4$ \\
	8   & $2^3$ & 0 & $C_2^3, C_4 \times C_2, C_8$ \\
	9   & $3^2 $ & 2  & $C_9$ \\
	16  & $2^4$ & 9 & $C_8 \times C_2$, $C_{16}$ \\
	25  & $5^2$ & 17 & $C_{25}$ \\
	27  & $3^3$ & 9 & $C_9 \times C_3$, $C_{27}$ \\
	32  & $2^5$ & 0 & $C_2^5, C_4 \times C_2^3, 
	C_4^2 \times C_2, C_8 \times C_2^2, C_8 \times C_4, C_{16} \times C_2, C_{32}$ \\
    49  & $7^2$ & 29 & $C_{49}$ \\
	64  & $2^6$ & 55& $C_{16} \times C_2^2$, $C_{16} \times C_4$, $C_{32} \times C_2$, $C_{64}$ \\
	81  & $3^4$ & 125 & $C_{27} \times C_3, C_{81}$ \\
   121  & $11^2$ & 43 & $C_{121}$ \\	   
   125  & $5^3$ & 38 & $C_{25} \times C_5, C_{125}$ \\
   128  & $2^7$ & 0 & $C_2^7, C_4 \times C_2^5, C_4^2 \times C_2^3, C_4 \times C_2^5, 
   C_8 \times C_2^4, C_8 \times C_4 \times C_2^2,
   C_8 \times C_4^2$,
    \\
   & & & $C_8^2 \times C_2, C_{16} \times C_2^3, C_{16} \times C_4 \times C_2, C_{16} \times C_8, C_{32} \times C_2^2,
   C_{32} \times C_4$, 
 \\
   & & &   $C_{64} \times C_2, C_{128}$ \\ 
 169  & $13^2$ & 64 & $C_{169}$ \\
 243  & $3^5$ & 30 & $C_{9} \times C_3^3$, $C_9 \times C_9 \times C_3$, $C_{27} 
 \times C_3^2, C_{27} \times C_9, C_{81} \times C_3, 
   C_{243}$ 
    \\ \bottomrule
\end{tabular}
\end{center}
\caption{All abelian B-groups of composite prime-power degree $d$ where $d \le 255$; $f(d)$ is
the number of primitive permutation groups of degree $d$ that are not $2$-transitive.}
\end{table}
	


\section*{Acknowledgements}

I thank Nick Gill for his answer to a 
MathOverflow question outlining
an alternative proof of Proposition~\ref{prop:affineOrFact}
 and Derek Holt for several helpful comments on
 this question (see \url{mathoverflow.net/questions/258434/}).
I thank John Britnell and Peter M.~Neumann for helpful comments.

\bigskip

\phantom{text}
\vspace*{-18pt}
\renewcommand{\MR}[1]{\relax}
\def\cprime{$'$} \def\Dbar{\leavevmode\lower.6ex\hbox to 0pt{\hskip-.23ex
  \accent"16\hss}D} \def\cprime{$'$}
\providecommand{\bysame}{\leavevmode\hbox to3em{\hrulefill}\thinspace}
\providecommand{\MR}{\relax\ifhmode\unskip\space\fi MR }
\providecommand{\MRhref}[2]{%
  \href{http://www.ams.org/mathscinet-getitem?mr=#1}{#2}
}
\providecommand{\href}[2]{#2}


\end{document}